\documentclass[11pt,a4paper]{amsart}

\usepackage[T1]{fontenc}
\usepackage[utf8]{inputenc}
\usepackage{lmodern}
\usepackage{a4wide}
\usepackage{amsmath,amssymb,mathtools}
\usepackage{xcolor}
\usepackage{url}
\usepackage{hyperref}

\hypersetup{
 colorlinks=true,
 linkcolor=blue!55!black,
 citecolor=blue!55!black,
 urlcolor=blue!55!black,
 pdftitle={Complex (1+)-injective Banach spaces are 1-injective},
 pdfauthor={Tomasz Kania and Grzegorz Lewicki},
 pdfsubject={Injective Banach spaces, complex Lindenstrauss spaces, weak intersection properties}
}

\theoremstyle{plain}
\newtheorem{maintheorem}{Theorem}

\newtheorem{theorem}{Theorem}[section]
\newtheorem{proposition}[theorem]{Proposition}
\newtheorem{lemma}[theorem]{Lemma}
\newtheorem{corollary}[theorem]{Corollary}
\theoremstyle{definition}
\newtheorem{definition}[theorem]{Definition}
\theoremstyle{remark}
\newtheorem{remark}[theorem]{Remark}

\numberwithin{equation}{section}

\newcommand{\C}{\mathbb C}
\newcommand{\R}{\mathbb R}
\newcommand{\N}{\mathbb N}
\newcommand{\D}{\mathbb D}
\newcommand{\ext}{\operatorname{ext}}
\newcommand{\dist}{\operatorname{dist}}
\newcommand{\id}{\operatorname{id}}
\newcommand{\Ree}{\operatorname{Re}}
\newcommand{\T}{\mathbb T}
\newcommand{\pair}[2]{\langle #1,#2\rangle}

\title[Complex $(1^+)$-injective Banach spaces]
{Complex $(1^+)$-injective Banach spaces are $1$-injective}

\author[T.~Kania]{Tomasz Kania}
\address[T.~Kania]{Mathematical Institute\\Czech Academy of Sciences\\
\v Zitn\'a 25 \\115 67 Praha 1\\Czech Republic  and  Institute of
Mathematics and Computer Science\\ Jagiellonian University\\
{\L}ojasiewicza 6, 30-348 Krak\'{o}w, Poland
}
\email{kania@math.cas.cz, tomasz.marcin.kania@gmail.com}
\thanks{RVO: 67985840.}

\author[G.~Lewicki]{Grzegorz Lewicki}
\address[G.~Lewicki]{Institute of Mathematics and Computer Science\\
Jagiellonian University\\{\L}ojasiewicza 6\\30-348 Krak\'ow\\Poland}
\email{Grzegorz.Lewicki@im.uj.edu.pl}

\subjclass[2020]{Primary 46B04; Secondary 46B25, 46E15, 46B03}
\keywords{injective Banach space, Lindenstrauss space, weak intersection
property, projection constant, extreme point}

\begin{document}
\raggedbottom
\begin{abstract}
We prove that every complex $(1^+)$-injective Banach space is
$1$-injective.  Together with Lindenstrauss's real endpoint theorem and
our results for $\lambda>1$, this shows that, over either scalar field,
the implication from $(\lambda^+)$-injectivity to $\lambda$-injectivity
holds for every Banach space if and only if $\lambda=1$.
The main geometric ingredient is the following relative result:
if $X$ is a closed subspace of a complex
$P_1$-space $Z$ and $\lambda(X,Z)=1$, then $X^*$ is isometric to an
$L_1$-space and $B_X$ has an extreme point.  Its proof combines stability
of intersections of scalar discs with simultaneous rescaling by means of
two tangent discs.  We also show that complex $L_1$-preduals without
extreme points in their unit balls can have bidual projection constants
arbitrarily close to $1$.
\end{abstract}
\maketitle

\section{Introduction}

Let $\lambda\geqslant1$, and fix the real or complex scalar field.
A Banach space $X$ is called \emph{$\lambda$-injective} if every
isometric embedding of $X$ into a Banach space $Y$ admits a projection
$P\colon Y\to X$ with $\|P\|\leqslant\lambda$.
It is \emph{$(\lambda^+)$-injective} if it is
$(\lambda+\varepsilon)$-injective for every $\varepsilon>0$.
All embeddings and projections are linear over the chosen field.
The distinction concerns attainment: projections with norms arbitrarily
close to $\lambda$ need not yield one of norm at most $\lambda$.
We refer to \cite[Section~2.5]{DalesEtAl2016} for the standard theory
of injective spaces.

For real Banach spaces, Lindenstrauss proved the following endpoint
theorem.

\begin{theorem}[Lindenstrauss
{\cite[Theorem~6.10, p.~82]{Lindenstrauss1964}}]
\label{thm:lindenstrauss-real-endpoint}
Every real $(1^+)$-injective Banach space is $1$-injective.
\end{theorem}

The cited theorem is stated for Lindenstrauss's $E_\lambda$
cell-intersection properties; the formulation above follows from
$P_\lambda\Rightarrow E_\lambda$ and $E_1=P_1$, established in the
memoir.  The complex analogue was posed explicitly by Dashiell
in a MathOverflow question \cite{Dashiell2013}.
For each $1<\lambda\leqslant2$, we constructed in
\cite{KaniaLewicki2023} a $(\lambda^+)$-injective Banach space which is
not $\lambda$-injective.  The complementary range $\lambda>2$ was
treated in \cite{KaniaLewicki2026}.  Both constructions apply over
either scalar field.  Thus the only remaining case was the complex
endpoint, which is the subject of the present paper.
The extension theory obtained by restricting the ambient spaces to be
separable is of a different nature; see the monograph
\cite{AvilesEtAl2016}.

\begin{maintheorem}\label{thm:main}
Every complex $(1^+)$-injective Banach space is $1$-injective.
\end{maintheorem}

The difficulty in passing from real to complex scalars lies in the
intersection geometry used in Lindenstrauss's proof.  Testing a family
of balls against real functionals produces closed scalar intervals.
Pairwise intersection of these intervals ensures a common point, and
therefore suffices when auxiliary balls are adjoined in the real facial
descent.  Complex functionals instead produce discs.  Already three
pairwise intersecting unit discs can have empty intersection, so the
same verification does not preserve the weak intersection property;
see Remark~\ref{rem:scalar-obstruction} for the precise comparison.
Our replacement constructs one scalar point common to every disc in
the enlarged family.  Radial clipping provides that point, and two
externally tangent discs then force the desired value of a chosen
functional.

There is also no immediate reduction to the underlying real space.
Even $\C$ is complex $1$-injective, whereas its underlying real
Euclidean plane is not real $1$-injective.  We first obtain an extreme
point and then use it to construct a real form to which Lindenstrauss's
theorem applies.

The new geometric result is most naturally stated in a relative form.
For a closed subspace $X$ of a Banach space $Z$, put
\[
 \lambda(X,Z)=\inf\{\|P\|:P\colon Z\to X
 \text{ is a projection}\},
\]
where $\inf\varnothing=+\infty$.

\begin{maintheorem}\label{thm:endpoint}
Let $Z$ be a complex $P_1$-space and let $X$ be a closed subspace of $Z$.
If $\lambda(X,Z)=1$, then $X^*$ is isometric to an $L_1$-space and
$B_X$ has an extreme point.
\end{maintheorem}

Combining Theorem~\ref{thm:main} with
Theorem~\ref{thm:lindenstrauss-real-endpoint} and
\cite{KaniaLewicki2023,KaniaLewicki2026} gives the following
classification.  Its assertion concerns the implication for all
Banach spaces over a fixed scalar field.

\begin{maintheorem}\label{thm:classification}
Let $\mathbb K\in\{\R,\C\}$ and $\lambda\geqslant1$.
The following assertions are equivalent:
\begin{enumerate}
\item[(i)] $\lambda=1$;
\item[(ii)] every $(\lambda^+)$-injective Banach space over $\mathbb K$
is $\lambda$-injective.
\end{enumerate}
For each $\lambda>1$, a counterexample to the implication in
\textup{(ii)} can be chosen as a closed subspace of
$\ell_\infty(\N,\mathbb K)$.
\end{maintheorem}

We outline the proof of Theorem~\ref{thm:endpoint}.
Hustad's theorem characterises complex
$L_1$-preduals by the intersection of finite weakly intersecting families
of balls.  We first prove that almost contractive complementation in a
$P_1$-space gives intersections for arbitrary such families with uniformly
bounded radii.  A radial contraction of the scalar unit disc, together
with the $n$-ball property of $M$-ideals, then permits a chosen family
of extreme functionals, all below one common threshold, to be rescaled
simultaneously.
Iteration at a sequence of increasing thresholds gives, for every
$x\in B_X$, elements $y,z\in B_X$ such that
\[
 x=\frac13y+\frac23z
 \quad\text{and}\quad
 F(y)\subseteq F(x)\cap B_X(3x,2),
\]
where $F(x)$ is the smallest algebraic face of $B_X$ containing $x$.
The order argument of Aronszajn and Panitchpakdi
\cite[Section~4]{AP1956}, which we include, yields an extreme point.

To deduce Theorem~\ref{thm:main}, embed a $(1^+)$-injective space into an
$\ell_\infty$-space and apply Theorem~\ref{thm:endpoint}.  The
Hirsberg--Lazar representation
\cite[Theorems~1 and~2, p.~141]{HirsbergLazar1973}
provides a real form.  Lindenstrauss's theorem applies to that real form,
and Hasumi's theorem \cite[Theorem~2]{Hasumi1958} completes the proof.
Finally, Proposition~\ref{prop:no-uniform-gap} shows that the strict
inequality $\lambda(X,X^{**})>1$ for complex $L_1$-preduals without
extreme points cannot be strengthened to a uniform lower bound greater
than $1$.

\section{Preliminaries}

All Banach spaces below are over $\C$, except where the real scalar field
is explicitly indicated.  A real or complex Banach space is called a
\emph{Lindenstrauss space} if its dual is isometric to an $L_1$-space.
The symbols $B_X$ and $S_X$ denote the closed
unit ball and unit sphere, respectively, and
$\pair{x^*}{x}$ denotes the duality between $X^*$ and $X$.  We write
\[
 B_X(a,r)=\{x\in X:\|x-a\|\leqslant r\},
 \qquad B_X=B_X(0,1).
\]
For $c\in\C$ and $r\geqslant0$, let
\[
 D(c,r)=\{z\in\C:|z-c|\leqslant r\}.
\]
As usual, $\D=\{z\in\C:|z|<1\}$ and
$\T=\{z\in\C:|z|=1\}$.

\begin{definition}
Let $I$ be an arbitrary index set.  A family
\[
 \mathcal B=\{B_X(a_i,r_i):i\in I\}
\]
has the \emph{weak intersection property}, abbreviated WIP, if for every
$f\in B_{X^*}$ the scalar discs
\[
 D(\pair{f}{a_i},r_i),\qquad i\in I,
\]
have a common point.  For each $i$ we have
$f(B_X(a_i,r_i))\subseteq D(\pair{f}{a_i},r_i)$, and the inclusion
can be strict.  In particular, the test does not require a common
point of the actual images of the balls.
\end{definition}

\begin{remark}\label{rem:scalar-obstruction}
Over $\R$, the corresponding definition uses the intervals
$[\pair{f}{a_i}-r_i,\pair{f}{a_i}+r_i]$.
In that case WIP is equivalent to pairwise intersection of the original
balls.  Indeed, pairwise intersection of the balls is equivalent to
$\|a_i-a_j\|\leqslant r_i+r_j$ for all $i,j$.
These inequalities imply pairwise intersection of the scalar intervals
for every $f\in B_{X^*}$.  A pairwise intersecting family of non-empty
closed bounded intervals has a common point: its finite subfamilies
intersect, and one fixed interval is compact.  Conversely, WIP gives
$|\pair{f}{a_i-a_j}|\leqslant r_i+r_j$ for every $f\in B_{X^*}$;
taking the supremum over $f$ yields the original inequalities.

The analogous implication fails over $\C$.  The three unit discs
\[
 D(0,1),\qquad D(2,1),\qquad D(1+i\sqrt3,1)
\]
intersect pairwise, since their centres are at mutual distance $2$.
The first two meet only at $1$, which has distance $\sqrt3>1$ from
the third centre, so their total intersection is empty.
Thus pairwise intersection does not imply WIP even in the
one-dimensional complex space $\C$, which is $1$-injective by the
complex Hahn--Banach theorem.  Viewed as balls in the underlying real
Euclidean plane, the same three discs violate the ball-intersection
characterisation of real $1$-injectivity \cite{Nachbin1950}.
This is the scalar obstruction that the simultaneous construction in
Section~\ref{sec:rescaling} must overcome.
\end{remark}

\begin{theorem}[Hustad {\cite[Theorem~4.9]{Hustad1974}}]
\label{thm:hustad}
For a complex Banach space $X$, the following are equivalent:
\begin{enumerate}
\item[(i)] $X^*$ is isometric to an $L_1$-space;
\item[(ii)] every finite WIP family of closed balls in $X$ has non-empty
intersection.
\end{enumerate}
\end{theorem}

The four-ball criterion is due to Lima
\cite{LimaIsrael1976}.  In the appendix to a second paper, he proved the
sharper implication that the three-ball WIP already forces the dual to
be an $L_1$-space \cite[Theorem~17, p.~340]{LimaScand1976}.  Together
with Theorem~\ref{thm:hustad}, this characterises complex $L_1$-preduals.
In Lima's notation, an $E(3)$-space is precisely a complex Banach space
in which every WIP family of three balls has non-empty intersection.
Roy later gave a short proof of the four-ball implication \cite{Roy1981}.
Recall that a complex $P_1$-space is a complex $1$-injective Banach
space.  Hustad proved that a complex Banach space is a $P_1$-space if and
only if every, not necessarily finite, WIP family of balls has non-empty
intersection \cite[Theorem~1]{Hustad1973}.

The following result is stated explicitly by Hustad
\cite[p.~284]{Hustad1974}, who attributes the real case to Grothendieck
and the complex case to Sakai.

\begin{theorem}\label{thm:grothendieck-sakai}
For a complex Banach space $X$, the dual $X^*$ is isometric to an
$L_1$-space if and only if $X^{**}$ is a complex $P_1$-space.
\end{theorem}

Grothendieck proved the real assertion
\cite[Proposition~1 and Theorem~1, pp.~553--554]{Grothendieck1955};
for the ingredients in the complex case, see
\cite[Corollary~1.13.3, p.~30, and Theorem~1.16.7,
pp.~41--42]{Sakai1971} and
\cite[Theorem~2]{Hasumi1958}.

The extreme points of the dual ball form a norming set:
\begin{equation}\label{eq:extreme-boundary-norming}
 \|x\|=\sup\{|\pair{e}{x}|:e\in\ext B_{X^*}\}
 \qquad(x\in X).
\end{equation}
Indeed, the Hahn--Banach theorem gives $f\in B_{X^*}$ such that, after
multiplication by a scalar in $\T$, $\pair{f}{x}=\|x\|$.  By the
Krein--Milman theorem there is a net $(f_\alpha)$ of convex combinations
of points of $\ext B_{X^*}$ which converges to $f$ in the weak$^*$ topology.
Consequently
\[
 |\pair{f_\alpha}{x}|
 \leqslant\sup\{|\pair{e}{x}|:e\in\ext B_{X^*}\}.
\]
Passing to the limit gives one inequality in
\eqref{eq:extreme-boundary-norming}; the other follows from
$\ext B_{X^*}\subseteq B_{X^*}$.

For a point $x$ of a convex set $K$, let $F_K(x)$ denote the smallest
algebraic face of $K$ containing $x$.  We shall use the characterisation
\begin{equation}\label{eq:algebraic-face-characterisation}
 y\in F_K(x)
 \quad\Longleftrightarrow\quad
 y\in K\ \text{ and }\ x=t y+(1-t)z
 \text{ for some }z\in K\text{ and }0<t<1.
\end{equation}
For completeness, let $G$ denote the set on the right.  Equivalently,
$y\in G$ if and only if $y\in K$ and $x+c(x-y)\in K$ for some $c>0$.  If
$y_1,y_2\in G$, choose $c_j>0$ such that
$x+c_j(x-y_j)\in K$, and put $c=\min\{c_1,c_2\}$.  Convexity along the
segments joining $x$ to $x+c_j(x-y_j)$ gives
$x+c(x-y_j)\in K$ for $j=1,2$.  Hence
\[
 x+c\bigl(x-(s y_1+(1-s)y_2)\bigr)
 =s\{x+c(x-y_1)\}+(1-s)\{x+c(x-y_2)\}\in K
\]
for $0\leqslant s\leqslant1$, and $G$ is convex.  Suppose that
$y\in G$, that $x=t y+(1-t)z$, and that
$y=s u+(1-s)v$, where $0<s,t<1$ and $u,v,z\in K$.  Then
\[
 x=tsu+(1-ts)
 \left(\frac{t(1-s)v+(1-t)z}{1-ts}\right),
\]
and the analogous formula with $u$ and $v$ interchanged shows that
$u,v\in G$.  Thus $G$ is a face, and plainly $x\in G$.  Every face containing $x$ contains
$G$, whence $G=F_K(x)$.  In particular,
$y\in F_K(x)$ implies $F_K(y)\subseteq F_K(x)$.  For $K=B_X$ we write
simply $F(x)$.

Recall that a projection $L$ on a Banach space is an $L$-projection if
\[
 \|u\|=\|Lu\|+\|u-Lu\| \qquad(u\in X).
\]
The range of an $L$-projection is called an $L$-summand.  A closed
subspace $J$ of $X$ is an $M$-ideal if $J^\perp$ is an $L$-summand of
$X^*$; see \cite[Definition~I.1.1(c)]{HWW1993}.

\begin{lemma}\label{lem:atomic-band}
Let $X\ne\{0\}$ be a complex Banach space whose dual is isometric to an
$L_1$-space.  If $e\in\ext B_{X^*}$, then $\C e$ is an $L$-summand of
$X^*$ and $\ker e$ is an $M$-ideal in $X$.
\end{lemma}

\begin{proof}
Since $X\ne\{0\}$, every extreme point of $B_{X^*}$ has norm one.
Transport $e$ through an isometry $X^*\cong L_1(\mu)$, and denote its
image by $g$.  Thus $\|g\|_1=1$.  Define a probability measure by
\[
 \nu(A)=\int_A |g|\,d\mu.
\]
If $0<\nu(A)<1$, then
\[
 g=\nu(A)\frac{\mathbf1_Ag}{\nu(A)}
   +(1-\nu(A))\frac{\mathbf1_{A^c}g}{1-\nu(A)}
\]
is a non-trivial convex decomposition of $g$ into two distinct elements
of the $L_1$-unit sphere.  The extremality of $g$ therefore implies that
$\nu(A)\in\{0,1\}$ for every measurable set $A$.

Put $S=\{g\ne0\}$ and $S_n=\{|g|\geqslant1/n\}$.  Since
$\nu(S_n)\in\{0,1\}$, $S_n\uparrow S$ modulo a null set and
$\nu(S)=1$, there is $n$ such that $\nu(S_n)=1$.  The argument below
shows that $\mu(S\setminus S_n)=0$; consequently
$\mu(S)=\mu(S_n)\leqslant
n\int_{S_n}|g|\,d\mu\leqslant n<\infty$.

If $A\subseteq S$ and $\nu(A)=0$, then
\[
 A=\bigcup_{n=1}^\infty
   \bigl(A\cap\{|g|\geqslant1/n\}\bigr)
\]
modulo a null set, and each member of this union has $\mu$-measure zero.
Thus $\mu(A)=0$.  If $\nu(A)=1$, the same argument applied to $S\setminus
A$ gives $\mu(S\setminus A)=0$.  Hence $S$ is a $\mu$-atom.  The band of
$L_1(\mu)$ supported on $S$ is one-dimensional and equals $\C g$.
Multiplication by $\mathbf1_S$ is the $L$-projection onto this band; see
\cite[Example~I.1.6(a)]{HWW1993}.  Transporting the projection back to
$X^*$ shows that $\C e$ is an $L$-summand.  Since
$(\ker e)^\perp=\C e$, the definition of an $M$-ideal gives the last
assertion.
\end{proof}

We use the following consequence of the finite $n$-ball property.  For
every positive integer $n$, if
$J$ is an $M$-ideal in $X$, the balls $B_X(x_j,r_j)$,
$1\leqslant j\leqslant n$, have a common point, and each of them meets
$J$, then
\begin{equation}\label{eq:n-ball-property}
 J\cap\bigcap_{j=1}^{n}B_X(x_j,r_j+\varepsilon)\ne\varnothing
 \qquad(\varepsilon>0).
\end{equation}
This is precisely the $n$-ball property in
\cite[Theorem~I.2.2(ii)]{HWW1993}; in particular, it includes $n=2$.
Ball-intersection characterisations of subspaces and their applications
to $L_1$-preduals were developed systematically by Lima
\cite{LimaTAMS1977}; we use the standard $M$-ideal formulation above.

\begin{definition}
A Banach space has the \emph{bounded weak intersection property} if every
WIP family of closed balls with uniformly bounded radii has non-empty
intersection.
\end{definition}

\begin{remark}
Every space with the bounded weak intersection property is a
Lindenstrauss space.  Indeed, the property applies in particular to
finite WIP families, so Theorem~\ref{thm:hustad}, (ii)$\Rightarrow$(i),
applies.  We shall use this observation without further comment.
\end{remark}

\section{Bounded weak intersections}

We begin with a scalar estimate.

\begin{lemma}\label{lem:disc-stability}
Let $I\ne\varnothing$, let $0\leqslant r_i\leqslant R$ for $i\in I$, and suppose
\[
 A=\bigcap_{i\in I}D(c_i,r_i)\ne\varnothing.
\]
If $\eta\geqslant0$ and
\[
 z\in\bigcap_{i\in I}D(c_i,r_i+\eta),
\]
then
\begin{equation}\label{eq:disc-stability}
 \dist(z,A)\leqslant\sqrt{2R\eta+\eta^2}.
\end{equation}
\end{lemma}

\begin{proof}
The set $A$ is non-empty, closed, convex and bounded, so there is a nearest
point $a\in A$ to $z$.  Put
\[
 d=|z-a|.
\]
There is nothing to prove if $d=0$.  Otherwise set $u=(z-a)/d$, so
$|u|=1$.

Fix $0<t<d$.  The point $a+tu$ is not in $A$, since otherwise its distance
from $z$ would be $d-t<d$.  Hence some $i_t\in I$ satisfies
\[
 |a+tu-c_{i_t}|>r_{i_t}.
\]
Write $c=c_{i_t}$, $r=r_{i_t}$, and
\[
 v=\Ree\bigl(u\,\overline{(a-c)}\bigr).
\]
Squaring the strict inequality gives
\begin{equation}\label{eq:disc-exclusion}
 |a-c|^2+2tv+t^2>r^2.
\end{equation}
Since
\[
 |v|\leqslant|a-c|\leqslant r\leqslant R,
\]
equation \eqref{eq:disc-exclusion} implies
\[
 v>-\frac t2,
 \qquad r^2-|a-c|^2<2tR+t^2.
\]
Because $z=a+du$, we infer
\begin{align}
 |z-c|^2
 &=|a-c|^2+2dv+d^2\\
 &>r^2-(2tR+t^2)-dt+d^2.\label{eq:disc-lower-bound}
\end{align}
On the other hand, $z\in D(c,r+\eta)$, so
\begin{equation}\label{eq:disc-upper-bound}
 |z-c|^2\leqslant(r+\eta)^2
 \leqslant r^2+2R\eta+\eta^2.
\end{equation}
Combining \eqref{eq:disc-lower-bound} and \eqref{eq:disc-upper-bound},
then letting $t\downarrow0$, yields
\[
 d^2\leqslant2R\eta+\eta^2.
\]
This is \eqref{eq:disc-stability}.
\end{proof}

Notice that the index $i_t$ may vary with $t$; neither compactness nor
countability of $I$ is used.

\begin{theorem}\label{thm:bwip}
Let $Z$ be a complex $P_1$-space, and let $X$ be a closed subspace of $Z$.
Suppose that $\lambda(X,Z)=1$.
Every WIP family $\{B_X(a_i,r_i):i\in I\}$ with uniformly bounded radii
has non-empty intersection in $X$.
\end{theorem}

\begin{proof}
If $I=\varnothing$, there is nothing to prove.  We therefore assume
$I\ne\varnothing$.
Fix $R>0$ with $r_i\leqslant R$ for every $i$.

\smallskip\noindent\emph{Step 1.}
The family, regarded as a family of balls in $Z$, has WIP.  Indeed, if
$g\in B_{Z^*}$, then $g|_X\in B_{X^*}$, and the required scalar discs
are unchanged.  By Hustad's characterisation of complex $P_1$-spaces
\cite[Theorem~1]{Hustad1973},
\[
 \bigcap_{i\in I}B_Z(a_i,r_i)\ne\varnothing.
\]

\smallskip\noindent\emph{Step 2.}
Choose positive numbers $\delta_n\downarrow0$ so fast that
$\delta_n\leqslant\sqrt2-1$ and, with
\[
 \rho_n=R\sqrt{2\delta_n+\delta_n^2},
\]
we have $\rho_n\leqslant R$ and
\begin{equation}\label{eq:summable-radii}
 \sum_{n=1}^{\infty}(1+\delta_{n+1})\rho_n<\infty.
\end{equation}
For each $n$, choose a projection
\[
 P_n\colon Z\to X,
 \qquad P_n|_X=\id_X,
 \qquad \|P_n\|\leqslant1+\delta_n.
\]

\smallskip\noindent\emph{Step 3.}
Set $\mathcal F_0=\{B_X(a_i,r_i):i\in I\}$.  Inductively, after the WIP
family $\mathcal F_{n-1}$ has been constructed, apply Step~1 to
$\mathcal F_{n-1}$ and choose
\[
 z_n\in\bigcap_{B_X(a,r)\in\mathcal F_{n-1}}B_Z(a,r),
 \qquad p_n=P_nz_n\in X.
\]
Define
\[
 \mathcal F_n=\mathcal F_{n-1}\cup\{B_X(p_n,\rho_n)\}.
\]
Because $\rho_n\leqslant R$, all radii occurring in every $\mathcal F_n$ remain
bounded by $R$.

We verify WIP.  Fix $f\in B_{X^*}$ and put
\[
 A_f=\bigcap_{B_X(a,r)\in\mathcal F_{n-1}}D(\pair{f}{a},r).
\]
This set is non-empty by the induction hypothesis.  Since $P_na=a$,
\[
 |\pair{f}{p_n}-\pair{f}{a}|
 \leqslant\|P_n(z_n-a)\|
 \leqslant(1+\delta_n)r
 \leqslant r+\delta_nR.
\]
The set $A_f$ is a closed subset of any one of its bounded discs and is
therefore compact, so the distance to it is attained.
Lemma~\ref{lem:disc-stability}, with $\eta=\delta_nR$, supplies
$q_f\in A_f$ with
\[
 |q_f-\pair{f}{p_n}|\leqslant
 R\sqrt{2\delta_n+\delta_n^2}=\rho_n.
\]
Thus $q_f$ lies in every old scalar disc and in the new one.  The enlarged
family remains WIP.

\smallskip\noindent\emph{Step 4.}
For $n\geqslant2$, the choice of $z_n$ gives
$\|z_n-p_{n-1}\|\leqslant\rho_{n-1}$.  Hence
\[
 \|p_n-p_{n-1}\|
 \leqslant(1+\delta_n)\rho_{n-1}.
\]
By \eqref{eq:summable-radii}, $(p_n)$ is Cauchy.  Let $p_n\to p\in X$.
For each original ball,
\[
 \|p_n-a_i\|
 =\|P_n(z_n-a_i)\|
 \leqslant(1+\delta_n)r_i.
\]
Letting $n\to\infty$ gives $\|p-a_i\|\leqslant r_i$ for every $i$.  Therefore
$p$ belongs to the original intersection.
\end{proof}

\begin{corollary}\label{cor:relative-bwip}
Under the hypotheses of Theorem~\ref{thm:bwip}, the dual $X^*$ is
isometric to an $L_1$-space.
\end{corollary}

\begin{proof}
Theorem~\ref{thm:bwip} applies in particular to finite WIP families.
The conclusion follows from the implication (ii)$\Rightarrow$(i) in
Theorem~\ref{thm:hustad}.
\end{proof}

\section{Simultaneous rescaling}\label{sec:rescaling}

The next lemma supplies a common scalar witness for all the auxiliary
balls used below.  This preserves the full weak intersection property
required over the complex field.

\begin{lemma}\label{lem:clipping}
Fix $0<p<1$ and put
\[
 d=\sqrt{1-p^2}.
\]
For $a\in\overline{\D}$ define
\[
 q_p(a)=
 \begin{cases}
  a,& |a|\leqslant p,\\[2mm]
  p\,a/|a|,& |a|>p.
 \end{cases}
\]
Then
\begin{equation}\label{eq:clipping-inclusion}
 D(0,1)\cap D(a,d)\subset D(q_p(a),d).
\end{equation}
Moreover, for $w_p(a)=p^{-1}q_p(a)$,
\begin{align}
 |w_p(a)|&\leqslant1,\label{eq:clip-unit-bound}\\
 |w_p(a)-2a|&\leqslant1,\label{eq:clip-double-bound}\\
 |w_p(a)-a|&\leqslant1-p,\label{eq:clip-displacement}\\
 \left|w_p(a)-\frac ap\right|&\leqslant\frac{1-p}{p}.
 \label{eq:clip-scaled-displacement}
\end{align}
\end{lemma}

\begin{proof}
By rotation invariance, write $a=s\in[0,1]$.  If $s\leqslant p$, then
$q_p(a)=a$ and the inclusion \eqref{eq:clipping-inclusion} is tautological.

Assume $s>p$, let $z=u+iv$ belong to the left-hand side of
\eqref{eq:clipping-inclusion}, and
suppose for a contradiction that $|z-p|>d$.  Since $|z-s|\leqslant d$,
\[
 |z-s|^2-|z-p|^2=(s-p)(s+p-2u)
\]
forces $u>(s+p)/2$.  Consequently
\begin{align*}
 |z|^2
 &=|z-p|^2+2pu-p^2\\
 &>d^2+p(s+p)-p^2
 =d^2+ps
 >d^2+p^2=1,
\end{align*}
contrary to $z\in D(0,1)$.  This proves \eqref{eq:clipping-inclusion}.

Inequality \eqref{eq:clip-unit-bound} is immediate.  For
\eqref{eq:clip-double-bound}, if $|a|\leqslant p$, then
\[
 |w_p(a)-2a|=|1/p-2|\,|a|\leqslant|1-2p|\leqslant1;
\]
if $|a|>p$, it equals $|1-2|a||\leqslant1$.  Similarly,
\[
 |w_p(a)-a|
 =\begin{cases}
  (1-p)|a|/p,&|a|\leqslant p,\\
  1-|a|,&|a|>p,
 \end{cases}
\]
which proves \eqref{eq:clip-displacement}.  Finally,
\eqref{eq:clip-scaled-displacement} is equivalent to
$|q_p(a)-a|\leqslant1-p$, which is again immediate in the two cases.
\end{proof}

\begin{lemma}\label{lem:lifts}
Let $X\ne\{0\}$ be a complex Banach space such that $X^*$ is isometric
to an $L_1$-space.  Let $e\in\ext B_{X^*}$, $x\in B_X$, and
$0<p<1$.  Write
$r=|\pair{e}{x}|$ and assume $r\leqslant p$.  Choose $\theta\in\T$ so that
$\pair{e}{x}=r\theta$ when $r>0$; when $r=0$, choose $\theta$ arbitrarily.  Put
\[
 d=\sqrt{1-p^2},
 \qquad \zeta_\pm=(r\pm id)\theta.
\]
For every sequence $0<\varepsilon_k\downarrow0$ there are
$v_{\pm,k}\in X$ such that
\begin{align}
 \pair{e}{v_{\pm,k}}&=\zeta_\pm,\notag\\
 \|v_{\pm,k}\|&\leqslant1+\varepsilon_k,\label{eq:lift-norm}\\
 \|v_{\pm,k}-x\|&\leqslant d+\varepsilon_k.\label{eq:lift-distance}
\end{align}
\end{lemma}

\begin{proof}
By Lemma~\ref{lem:atomic-band}, $J=\ker e$ is an $M$-ideal in $X$.

Since $X\ne\{0\}$, the functional $e$ has norm one.  Define
\[
 \Phi\colon X/J\longrightarrow\C,
 \qquad \Phi(u+J)=\pair{e}{u}.
\]
This is a well-defined surjection.  Moreover,
$(X/J)^*=J^\perp=\C e$ and $\|e\|=1$, so the quotient norm formula gives
\[
 \|u+J\|=\sup_{|\alpha|\leqslant1}
   |\pair{\alpha e}{u}|=|\pair{e}{u}|.
\]
Thus $\Phi$ is an isometry.

Choose $u_\pm\in X$ with $\pair{e}{u_\pm}=\zeta_\pm$.  The preceding
identification of $X/J$ with $\C$ gives
\begin{equation}\label{eq:quotient-distances}
 \dist(u_\pm,J)=|\zeta_\pm|\leqslant1,
 \qquad
 \dist(u_\pm-x,J)=|\zeta_\pm-\pair{e}{x}|=d.
\end{equation}
The first inequality follows from
$|\zeta_\pm|^2=r^2+d^2\leqslant p^2+(1-p^2)=1$.
Set $\eta_k=\varepsilon_k/2$.  By \eqref{eq:quotient-distances}, each of the balls
\[
 B_X(u_\pm,1+\eta_k),
 \qquad B_X(u_\pm-x,d+\eta_k)
\]
meets $J$ by the definition of distance as an infimum.  The two balls
also intersect, since the distance between their centres is at most
$\|x\|\leqslant1<1+d+2\eta_k$.  Apply
\eqref{eq:n-ball-property} with $n=2$ and $\varepsilon=\eta_k$ to obtain
$j_{\pm,k}\in J$ satisfying
\[
 \|u_\pm-j_{\pm,k}\|\leqslant1+\varepsilon_k,
 \qquad
 \|u_\pm-x-j_{\pm,k}\|\leqslant d+\varepsilon_k.
\]
Set $v_{\pm,k}=u_\pm-j_{\pm,k}$.  Since
$\pair{e}{j_{\pm,k}}=0$, all three
assertions follow.
\end{proof}

\begin{theorem}\label{thm:saturation}
Let $X\ne\{0\}$ be a complex Banach space with the bounded weak
intersection property.  Let $x\in B_X$, fix $0<p<1$, and let
\[
 \mathcal S\subseteq\{e\in\ext B_{X^*}:|\pair{e}{x}|\leqslant p\}
\]
be arbitrary, possibly uncountable and possibly containing functionals for
which $\pair{e}{x}=0$.  Then there is $y\in B_X$ such that
\begin{align}
 \pair{e}{y}&=\frac{\pair{e}{x}}p
     &&(e\in\mathcal S),\label{eq:saturation-values}\\
 \|y-x\|&\leqslant1-p,\label{eq:saturation-displacement}\\
 \|py-x\|&\leqslant1-p,\label{eq:saturation-scaled-displacement}\\
 \|y-2x\|&\leqslant1.\label{eq:saturation-double-displacement}
\end{align}
In particular, there are $z,z'\in B_X$ with
\begin{equation}\label{eq:saturation-decompositions}
 x=p y+(1-p)z
 \qquad\text{and}\qquad
 x=\frac12y+\frac12z'.
\end{equation}
Consequently $y$ lies in the smallest algebraic face of $B_X$ containing
$x$.
\end{theorem}

\begin{proof}
Put $d=\sqrt{1-p^2}$ and choose, once and for all, a sequence
$0<\varepsilon_k\downarrow0$ with $\varepsilon_k\leqslant1$.  Let
\[
 \rho_k=\sqrt{2\varepsilon_k+\varepsilon_k^2}.
\]
For every $e\in\mathcal S$, choose the two families
$v_{e,\pm,k}$ furnished by Lemma~\ref{lem:lifts}.  Consider the balls
\begin{align}
 &B_X(0,1),\qquad B_X(2x,1),\notag\\
 &B_X(x,1-p),\qquad
 B_X\left(\frac{x}{p},\frac{1-p}{p}\right),
 \label{eq:saturation-ball-family}\\
 &B_X\left(\frac{v_{e,\pm,k}}p,
            \frac{d+\rho_k}{p}\right)
 \quad(e\in\mathcal S,\ \pm,\ k\in\N).\notag
\end{align}
Their radii are uniformly bounded.  We claim that the whole family has
WIP.

Fix $f\in B_{X^*}$, put $a=\pair{f}{x}$, and take as scalar witness
\[
 w=w_p(a)=p^{-1}q_p(a).
\]
Inequalities \eqref{eq:clip-unit-bound}--
\eqref{eq:clip-scaled-displacement} put $w$ in the four scalar discs
associated with the first four balls in \eqref{eq:saturation-ball-family}.

For an auxiliary centre $v=v_{e,\pm,k}$,
\eqref{eq:lift-norm}--\eqref{eq:lift-distance} give
\[
 \pair{f}{v}\in D(0,1+\varepsilon_k)
       \cap D(a,d+\varepsilon_k).
\]
The exact lens $D(0,1)\cap D(a,d)$ is non-empty.  Apply
Lemma~\ref{lem:disc-stability} to its two discs with $R=1$ and
$\eta=\varepsilon_k$.  There is therefore
$b_k\in D(0,1)\cap D(a,d)$ such that
\[
 |\pair{f}{v}-b_k|\leqslant\rho_k.
\]
By Lemma~\ref{lem:clipping}, $b_k\in D(q_p(a),d)$, and hence
\[
 |\pair{f}{v}-q_p(a)|
 \leqslant|\pair{f}{v}-b_k|+|b_k-q_p(a)|
 \leqslant\rho_k+d.
\]
After division by $p$, this says exactly that $w$ belongs to the scalar
disc associated with the auxiliary ball.  Thus the entire family is WIP.

The bounded weak intersection property supplies $y\in X$ in every ball.
The first four balls give
$y\in B_X$ and \eqref{eq:saturation-displacement}--
\eqref{eq:saturation-double-displacement}.
Equation \eqref{eq:saturation-scaled-displacement} gives the first
decomposition in \eqref{eq:saturation-decompositions}, while
\eqref{eq:saturation-double-displacement} gives the second.

It remains to prove \eqref{eq:saturation-values}.  Fix $e\in\mathcal S$,
write $\pair{e}{x}=r\theta$, with the convention in
Lemma~\ref{lem:lifts}, and recall
$\zeta_\pm=(r\pm id)\theta$.  Membership in all the corresponding
auxiliary balls gives
\[
 |p\pair{e}{y}-\zeta_\pm|\leqslant d+\rho_k
 \qquad(k\in\N).
\]
Letting $k\to\infty$,
\[
 |p\pair{e}{y}-\zeta_\pm|\leqslant d.
\]
The two radius-$d$ discs centred at $\zeta_+$ and $\zeta_-$ are externally
tangent; their only common point is their midpoint
$r\theta=\pair{e}{x}$.  Hence
$p\pair{e}{y}=\pair{e}{x}$, proving
\eqref{eq:saturation-values}, including when $\pair{e}{x}=0$.

Finally, \eqref{eq:algebraic-face-characterisation}, applied to either
strict convex decomposition in \eqref{eq:saturation-decompositions},
shows that $y\in F(x)$.
\end{proof}

\section{Facial descent and proof of Theorem~\ref{thm:endpoint}}

\begin{lemma}\label{lem:strict-descent}
Let $X\ne\{0\}$ be a complex Banach space with the bounded weak
intersection property.  If $x\in B_X$ is not extreme, then there is
$y\in F(x)$ with
\[
 F(y)\subsetneq F(x).
\]
\end{lemma}

\begin{proof}
Choose $v\ne0$ with $x\pm v\in B_X$.  By
\eqref{eq:extreme-boundary-norming}, there is $e\in\ext B_{X^*}$ with
$\pair{e}{v}\ne0$.  For $-1<t<1$, put $x_t=x+tv$.  We have
\[
 x=\frac12(x-v)+\frac12(x+v),\qquad
 x_t=\frac{1-t}{2}(x-v)+\frac{1+t}{2}(x+v).
\]
All coefficients are positive.  Consequently every face containing
either $x$ or $x_t$ contains both endpoints of the chord, and
$F(x_t)=F(x)$.

Put $a=\pair{e}{x}$ and $b=\pair{e}{v}\ne0$.  The equation
$a+tb=0$ has at most one real solution.  Moreover,
\[
 |a+tb|^2=|b|^2t^2+2\Ree(a\overline b)t+|a|^2,
\]
so $|a+tb|=1$ is a non-constant quadratic equation and has at most two
real solutions.  Choose $t\in(-1,1)$ outside these finite exceptional
sets.  Since $x_t\in B_X$, this choice gives
\begin{equation}\label{eq:intermediate-modulus}
 0<|\pair{e}{x_t}|<1.
\end{equation}

Set $p=|\pair{e}{x_t}|$ and apply Theorem~\ref{thm:saturation} to $x_t$ and
$\mathcal S=\{e\}$.  It gives $y\in F(x_t)=F(x)$ with
\[
 |\pair{e}{y}|=1.
\]
The set
\[
 \{u\in B_X:\pair{e}{u}=\pair{e}{y}\}
\]
is a face containing $y$, hence it contains $F(y)$.  It does not contain
$x_t$ by \eqref{eq:intermediate-modulus}.  Therefore $F(y)$ cannot equal
$F(x_t)$, and the inclusion is strict.
\end{proof}

The next iteration retains a positive coefficient of the limiting point.

\begin{lemma}\label{lem:recursion}
Set $m_0=1/3$.  Given $0<m_n<1$, let $p_n\in(0,1)$ be the positive solution
of
\[
 \frac{m_n}{p_n}=2p_n-1,
\]
and define
\begin{equation}\label{eq:recursion-step}
 m_{n+1}=2p_n-1=\frac{m_n}{p_n}.
\end{equation}
Then
\begin{equation}\label{eq:recursion-limits}
 m_n\uparrow1,
 \qquad
 \prod_{k=0}^{n-1}p_k=\frac{m_0}{m_n}\longrightarrow\frac13,
 \qquad
 \sum_{n=0}^\infty(1-p_n)<\infty.
\end{equation}
\end{lemma}

\begin{proof}
Explicitly,
\[
 p_n=\frac{1+\sqrt{1+8m_n}}4.
\]
For $0<m_n<1$, one has $m_n<p_n<1$, so
$m_n<m_{n+1}=m_n/p_n<1$.  Thus $(m_n)$ is increasing and bounded, and
has a limit in $(0,1]$.  This limit must solve
$m=(\sqrt{1+8m}-1)/2$, whose positive solution is $1$.  Thus $m_n\uparrow1$.
The product telescopes by \eqref{eq:recursion-step}.  Finally, for every $N$,
\[
 \sum_{n=0}^{N-1}(1-p_n)
 \leqslant\sum_{n=0}^{N-1}(-\log p_n)
 =-\log\prod_{n=0}^{N-1}p_n
 =\log\frac{m_N}{m_0}\leqslant\log3.
\]
\end{proof}

\begin{theorem}
\label{thm:high-saturation}
Let $X\ne\{0\}$ be a complex Banach space with the bounded weak
intersection property.  Fix $x_0\in B_X$ and put
\[
 \mathcal H=\{e\in\ext B_{X^*}:|\pair{e}{x_0}|>1/3\}.
\]
There exist $y,z\in B_X$ such that
\begin{align}
 |\pair{e}{y}|&=1 &&(e\in\mathcal H),
       \label{eq:high-saturation-values}\\
 x_0&=\frac13y+\frac23z,\label{eq:high-saturation-decomposition}\\
 F(y)&\subseteq F(x_0),\label{eq:high-saturation-face}\\
 F(y)&\subseteq B_X(3x_0,2).\label{eq:high-saturation-ball}
\end{align}
\end{theorem}

\begin{proof}
Let $(m_n)$ and $(p_n)$ be as in Lemma~\ref{lem:recursion}.  We construct
$x_n\in B_X$ inductively so that
\begin{equation}\label{eq:iteration-invariant}
 |\pair{e}{x_n}|\geqslant m_n\qquad(e\in\mathcal H).
\end{equation}
This holds for $n=0$ by the definition of $\mathcal H$.

At stage $n$, put
\[
 \mathcal S_n
 =\{e\in\mathcal H:|\pair{e}{x_n}|\leqslant p_n\}.
\]
If $\mathcal S_n$ is non-empty, apply Theorem~\ref{thm:saturation} to
$x_n,p_n,\mathcal S_n$, call the output $x_{n+1}$, and choose
$z_n\in B_X$ from the first decomposition in
\eqref{eq:saturation-decompositions}.  If $\mathcal S_n$ is empty, take
$x_{n+1}=x_n$ and $z_n=x_n$; all the conclusions used below still hold,
because
\[
 \|x_{n+1}-x_n\|=0\leqslant1-p_n,
 \qquad
 x_n=p_nx_{n+1}+(1-p_n)z_n.
\]
In either case,
\begin{align}
 \|x_{n+1}-x_n\|&\leqslant1-p_n,\label{eq:iteration-step-bound}\\
 x_n&=p_nx_{n+1}+(1-p_n)z_n
       \quad\text{for some }z_n\in B_X.\label{eq:iteration-decomposition}
\end{align}

We verify \eqref{eq:iteration-invariant} at the next stage.  If
$e\in\mathcal S_n$, then
\[
 |\pair{e}{x_{n+1}}|
 =\frac{|\pair{e}{x_n}|}{p_n}
 \geqslant\frac{m_n}{p_n}=m_{n+1}.
\]
If $e\in\mathcal H\setminus\mathcal S_n$, then
$|\pair{e}{x_n}|>p_n$, and \eqref{eq:iteration-step-bound} gives
\[
 |\pair{e}{x_{n+1}}|
 \geqslant |\pair{e}{x_n}|-(1-p_n)
 >2p_n-1=m_{n+1}.
\]
Thus the induction closes.

By \eqref{eq:recursion-limits} and \eqref{eq:iteration-step-bound},
$(x_n)$ is norm Cauchy.  Let $x_n\to y\in B_X$.  Since $m_n\to1$,
\eqref{eq:iteration-invariant} implies
\eqref{eq:high-saturation-values}.

Let
\[
 a_N=\prod_{n=0}^{N-1}p_n.
\]
Repeatedly substituting \eqref{eq:iteration-decomposition} gives
\begin{equation}\label{eq:iterated-decomposition}
 x_0=a_Nx_N+
 \sum_{k=0}^{N-1}
 \left(\prod_{j=0}^{k-1}p_j\right)(1-p_k)z_k.
\end{equation}
The non-negative coefficients in the sum add to $1-a_N$.  Since
$a_N\to1/3$, the series
\[
 s=\sum_{k=0}^{\infty}
 \left(\prod_{j=0}^{k-1}p_j\right)(1-p_k)z_k
\]
converges absolutely in norm and satisfies $\|s\|\leqslant2/3$.  Define
$z=(3/2)s\in B_X$.  Since $a_Nx_N\to y/3$, letting $N\to\infty$ in
\eqref{eq:iterated-decomposition} proves
\eqref{eq:high-saturation-decomposition}.
In particular, $y\in F(x_0)$ and hence $F(y)\subseteq F(x_0)$, proving
\eqref{eq:high-saturation-face}.

It remains to prove \eqref{eq:high-saturation-ball}.  Fix $w\in F(y)$
and $e\in\ext B_{X^*}$.  If $e\in\mathcal H$, then
\eqref{eq:high-saturation-values} and
the extremality of a point of the scalar unit disc imply
\[
 \pair{e}{w}=\pair{e}{y}.
\]
Indeed, write $y=t w+(1-t)v$ with $v\in B_X$ and $0<t<1$, and apply $e$;
the unimodular point $\pair{e}{y}$ cannot be a non-trivial convex combination of
two different points of the scalar unit disc.
Using \eqref{eq:high-saturation-decomposition},
\[
 \pair{e}{w-3x_0}=\pair{e}{y}-\pair{e}{y}-2\pair{e}{z}
 =-2\pair{e}{z},
\]
so $|\pair{e}{w-3x_0}|\leqslant2$.  If $e\notin\mathcal H$, then
$|\pair{e}{x_0}|\leqslant1/3$, and the elementary estimate
\[
 |\pair{e}{w-3x_0}|\leqslant
 |\pair{e}{w}|+3|\pair{e}{x_0}|\leqslant2
\]
applies.  Taking the supremum over the norming set $\ext B_{X^*}$ proves
$\|w-3x_0\|\leqslant2$.
\end{proof}

Notice that \eqref{eq:high-saturation-decomposition} does not require any
alignment between the phases of $\pair{e}{y}$ and $\pair{e}{x_0}$.

We shall use the following elementary identity of Aronszajn and
Panitchpakdi.

\begin{lemma}\label{lem:face-cap}
For every $x\in B_X$ and $\alpha>0$,
\begin{equation}\label{eq:face-cap-identity}
 B_X\cap B_X((1+\alpha)x,\alpha)
 =F(x)\cap\bigl[-\alpha F(x)+(1+\alpha)x\bigr].
\end{equation}
\end{lemma}

\begin{proof}
Suppose $y$ belongs to the left-hand side.  Then
\[
 y=-\alpha z+(1+\alpha)x
\]
for some $z\in B_X$, and
\[
 x=\frac1{1+\alpha}y+
   \frac{\alpha}{1+\alpha}z.
\]
This is a strict convex combination.  Hence both $y$ and $z$ belong to
$F(x)$, proving inclusion in the right-hand side.  The reverse inclusion is
immediate.  Identity \eqref{eq:face-cap-identity} is the norm-ball identity
used in the cell construction of Aronszajn and Panitchpakdi;
cf.\ \cite[Section~4, Theorem~1 and equation~(4.1)]{AP1956}.  The
argument above is included to fix the normalisation needed here.
\end{proof}

\begin{proposition}\label{prop:successor}
Let $X\ne\{0\}$ be a complex Banach space with the bounded weak
intersection property.  If $x\in B_X$ is not an extreme point, then
there is $u\in B_X$
such that
\begin{equation}\label{eq:proper-descendant}
 F(u)\subsetneq F(x)
 \qquad\text{and}\qquad
 F(u)\subseteq B_X(3x,2).
\end{equation}
\end{proposition}

\begin{proof}
Apply Theorem~\ref{thm:high-saturation} to obtain $y$ with
\begin{equation}\label{eq:initial-descendant}
 F(y)\subseteq F(x)\cap B_X(3x,2).
\end{equation}
If $F(y)\subsetneq F(x)$, take $u=y$.

If $F(y)=F(x)$, then $y$ is not an extreme point because $x$ is not.
Apply
Lemma~\ref{lem:strict-descent} at $y$ to obtain $u$ with
$F(u)\subsetneq F(y)$.  Since $F(u)\subseteq F(y)$,
\eqref{eq:initial-descendant} still gives the second inclusion in
\eqref{eq:proper-descendant}.
\end{proof}

The following proof is adapted from the order argument of Aronszajn and
Panitchpakdi; cf.\ \cite[Section~4, Theorem~2]{AP1956}.

\begin{proof}[Proof of Theorem~\ref{thm:endpoint}]
The equality $\lambda(X,Z)=1$ implies $X\ne\{0\}$, since
$\lambda(\{0\},Z)=0$.  By Corollary~\ref{cor:relative-bwip}, $X^*$ is
isometric to an $L_1$-space, and Theorem~\ref{thm:bwip} says that $X$ has
the bounded weak intersection property.

Define a relation $\preceq$ on $B_X$ as follows.  We put $y\preceq x$ if
$y=x$, or if
\[
 F(y)\subsetneq F(x)
 \quad\text{and}\quad
 F(y)\subseteq B_X(3x,2).
\]
Write $y\prec x$ when $y\preceq x$ and $y\ne x$.
This is a partial order.  Antisymmetry follows from strict inclusion of
faces when the points differ.  For transitivity, if $z\prec y\prec x$, then
$F(z)\subseteq F(y)\subsetneq F(x)$ and
$F(z)\subseteq B_X(3x,2)$.

We show that every chain $\mathcal A\subset B_X$ has a lower bound.  If
$\mathcal A=\varnothing$, use $0$.  We may therefore assume that
$\mathcal A$ is non-empty.  If the chain has a least member, use it.
Otherwise consider the ball family
\begin{equation}\label{eq:chain-ball-family}
 \mathcal G
 =\{B_X(0,1)\}
  \cup\{B_X(3z,2):z\in\mathcal A\}.
\end{equation}
Every finite subfamily has a common point.  Indeed, if balls indexed by
a non-empty finite subchain have been selected, that finite subset of the
chain has a least element; denote it by $z_0$.  Then $z_0\in B_X$.  For
any selected $z$, either
$z=z_0$, in which case
\[
 \|z_0-3z_0\|=2\|z_0\|\leqslant2,
\]
or $z_0\prec z$, in which case
$z_0\in F(z_0)\subseteq B_X(3z,2)$.  Thus $z_0$ belongs to every selected
ball, including $B_X(0,1)$ if it is present.  If the finite subfamily
consists only of $B_X(0,1)$, use $0$.  Hence, for every $f\in B_{X^*}$, the scalar
discs associated with \eqref{eq:chain-ball-family} have the finite
intersection property.  Since
the compact disc $D(0,1)$ is one of them, all the scalar discs have a common
point.  Thus \eqref{eq:chain-ball-family} has WIP.  Its radii are bounded
by $2$, so the bounded weak intersection property gives
\begin{equation}\label{eq:chain-intersection-point}
 y_0\in B_X\cap\bigcap_{z\in\mathcal A}B_X(3z,2).
\end{equation}

Fix $z\in\mathcal A$.  Apply Lemma~\ref{lem:face-cap} with $\alpha=2$.
By \eqref{eq:chain-intersection-point}, the identity gives $y_0\in F(z)$.  Since the chain has no least member, choose
$z'\in\mathcal A$ with $z'\prec z$.  Applying the same identity, again
with $\alpha=2$, at $z'$ gives $y_0\in F(z')$.  Therefore
\begin{equation}\label{eq:chain-lower-bound-face}
 F(y_0)\subseteq F(z')
 \subseteq F(z)\cap B_X(3z,2).
\end{equation}
Since $F(y_0)\subseteq F(z')\subsetneq F(z)$,
\eqref{eq:chain-lower-bound-face} shows that
$y_0\preceq z$ for every $z\in\mathcal A$.

The dual form of the Kuratowski--Zorn lemma now supplies a minimal element
$x_0$ of $B_X$.  If $x_0$ were not an extreme point,
Proposition~\ref{prop:successor} would give $u\prec x_0$, contradicting
minimality.  Thus $x_0\in\ext B_X$.
\end{proof}

\begin{corollary}\label{cor:bidual-endpoint}
Let $X$ be a complex Banach space such that $X^*$ is isometric to an
$L_1$-space.  If $\lambda(X,X^{**})=1$, then
$\ext B_X\ne\varnothing$.
\end{corollary}

\begin{proof}
By Theorem~\ref{thm:grothendieck-sakai}, $X^{**}$ is a complex
$P_1$-space.  The conclusion follows from Theorem~\ref{thm:endpoint},
applied with $Z=X^{**}$.
\end{proof}

The contrapositive of Corollary~\ref{cor:bidual-endpoint} says that a
complex $L_1$-predual without extreme points in its unit ball satisfies
$\lambda(X,X^{**})>1$.  The following construction shows that this
inequality admits no uniform improvement.

\begin{proposition}\label{prop:no-uniform-gap}
For every $\varepsilon>0$ there is a complex Banach space $X$ such
that $X^*$ is isometric to an $L_1$-space, $\ext B_X=\varnothing$, and
\[
 1<\lambda(X,X^{**})\leqslant1+\varepsilon.
\]
\end{proposition}

\begin{proof}
Fix $0<\alpha<1$ and a free ultrafilter
$p\in\beta\N\setminus\N$.  Identifying $\ell_\infty$ with
$C(\beta\N,\C)$, write $\delta_p(f)=\lim_p f(n)$, and put
\[
 Z=\ell_\infty\oplus_\infty\C,
 \qquad
 X_\alpha=\{(f,z)\in Z:\delta_p(f)=\alpha z\}.
\]
This is a closed hyperplane of $Z$, with its inherited norm.

Let $M(\beta\N)$ denote the space of regular complex Borel measures,
equipped with the total variation norm, and set
$M_0=\{\mu\in M(\beta\N):\mu(\{p\})=0\}$.
The singleton $\{p\}$ gives the band decomposition
$M(\beta\N)=M_0\oplus_1\C\delta_p$.
Since $X_\alpha^\perp=\C(\delta_p,-\alpha)$ in
$Z^*=M(\beta\N)\oplus_1\C$, the quotient norm gives
\[
 X_\alpha^*
 \cong M_0\oplus_1
       \bigl(\ell_1^2/\C(1,-\alpha)\bigr)
 \cong M_0\oplus_1\C
\]
isometrically.  To check the last identification, the map
$[(a,b)]\mapsto\alpha a+b$ is well-defined and has kernel zero.
The inequality $|\alpha a+b|\leqslant|a|+|b|$ gives the lower bound
for the quotient norm, and the representative $(0,\alpha a+b)$
attains it.  As a band in the $L_1$-space $M(\beta\N)$, the space
$M_0$ is itself an $L_1$-space.  Consequently, so is $X_\alpha^*$.

If $(f,z)\in B_{X_\alpha}$, then
$|\delta_p(f)|=\alpha|z|\leqslant\alpha<1$.
There is therefore an $n\in\N$ with $|f(n)|<1$: otherwise the
ultrafilter limit of the unimodular numbers $f(n)$ would also be
unimodular.  Choose $0<t<1-|f(n)|$, and let $e_n$ be the $n$th
unit vector in $\ell_\infty$.
Since $\delta_p(e_n)=0$, the two distinct points
$(f+te_n,z)$ and $(f-te_n,z)$ lie in $B_{X_\alpha}$ and have
midpoint $(f,z)$.  Thus $\ext B_{X_\alpha}=\varnothing$.

For each $n\in\N$, define $a_n\in X_\alpha^*$ by
$\pair{a_n}{(f,z)}=f(n)$, so $\|a_n\|\leqslant1$.
For $u^{**}\in X_\alpha^{**}$, put
\[
 S(u^{**})=\bigl(\pair{u^{**}}{a_n}\bigr)_{n\in\N},
 \qquad
 Q(u^{**})=
 \bigl(S(u^{**}),\alpha^{-1}\delta_p(S(u^{**}))\bigr).
\]
Then $\|S(u^{**})\|_\infty\leqslant\|u^{**}\|$, and $Q$ is a
linear map into $X_\alpha$ with $\|Q\|\leqslant\alpha^{-1}$.
On the canonical copy of $X_\alpha$, we have $S(f,z)=f$ and
$\alpha^{-1}\delta_p(f)=z$, so $Q$ fixes $X_\alpha$.
Hence $Q$ is a projection and
$\lambda(X_\alpha,X_\alpha^{**})\leqslant\alpha^{-1}$.
The strict lower bound follows from
Corollary~\ref{cor:bidual-endpoint}.
Taking $\alpha=(1+\varepsilon)^{-1}$ completes the proof.
\end{proof}

\section{Real forms and proof of Theorem~\ref{thm:main}}

We recall that the Taylor complexification of a real Banach space $E$ is
the complex space $E_{\C}=E\oplus iE$ with norm
\begin{equation}\label{eq:taylor}
 \|a+ib\|_{E_{\C}}
 =\sup_{\theta\in\R}
   \|a\cos\theta-b\sin\theta\|_E
 \qquad(a,b\in E).
\end{equation}
If $T\colon E\to F$ is a bounded real-linear operator, then, for
$a,b\in E$,
\begin{align*}
 \|T_{\C}(a+ib)\|_{F_{\C}}
 &=\sup_{\theta\in\R}
   \|T(a)\cos\theta-T(b)\sin\theta\|_F\\
 &=\sup_{\theta\in\R}
   \|T(a\cos\theta-b\sin\theta)\|_F\\
 &\leqslant \|T\|\sup_{\theta\in\R}
   \|a\cos\theta-b\sin\theta\|_E
 =\|T\|\,\|a+ib\|_{E_{\C}}.
\end{align*}
Thus $\|T_{\C}\|\leqslant\|T\|$; restriction to the real copy gives the
reverse inequality.  If $T$ is an isometry, then
\[
 \|T(a\cos\theta-b\sin\theta)\|_F
 =\|a\cos\theta-b\sin\theta\|_E
 \qquad(\theta\in\R),
\]
so the two suprema agree and $T_{\C}$ is an isometry.  The operator-norm assertion is
\cite[Proposition~4]{MunozEtAl1999}; see \cite{Taylor1943} for Taylor's
original discussion.

We shall use the following two results of Hirsberg and Lazar
\cite[Theorems~1 and~2, p.~141]{HirsbergLazar1973}.  If $Y$ is a complex
Lindenstrauss space, $u\in\ext B_Y$, and
\[
 S_u=\{y^*\in Y^*: \|y^*\|=\pair{y^*}{u}=1\},
\]
then $S_u$ is weak$^*$ compact and the canonical map
\[
 J_u\colon Y\longrightarrow C(S_u,\C),\qquad
 (J_uy)(s)=\pair{s}{y},
\]
is an isometry; moreover, $J_uu=\mathbf1$.  Among the four equivalent
conditions in their second theorem are that a closed function space
$A\subseteq C(K,\C)$ is a complex Lindenstrauss space and that $A$ is
self-adjoint with $\Ree A$ a real Lindenstrauss space.  Here a function
space contains the constants and separates the points of $K$.
The associated homogeneous function-space classification of complex
Lindenstrauss spaces was developed by Olsen \cite{Olsen1974}; see also
Effros's structural characterisation \cite[Theorem~4.3]{Effros1974}.

\begin{lemma}\label{lem:hirsberg-lazar-real-form}
Let $X\ne\{0\}$ be a complex Banach space such that $X^*$ is isometric
to an $L_1$-space.  If $\ext B_X\ne\varnothing$, then there is a real
Banach space $E$ such that $X$ is complex-linearly isometric to
the Taylor complexification $E_{\C}$.
\end{lemma}

\begin{proof}
Choose $u\in\ext B_X$.  Since $X\ne\{0\}$, we have $\|u\|=1$.  Let
$S=S_u$ and let $J=J_u$ be the canonical map above.  By the first
Hirsberg--Lazar theorem, $J$ is an isometry and $Ju=\mathbf1$.  Set
$A=J(X)$.  Then $A$ is a closed complex-linear subspace of $C(S,\C)$,
it contains the constants, and it separates the points of $S$.  It is
therefore a function space.  Since $A$ is isometric to $X$, it is a
complex Lindenstrauss space.  The second Hirsberg--Lazar theorem shows
that $A$ is self-adjoint.

Put $E=A\cap C(S,\R)$.  Since $A$ is self-adjoint, the real and imaginary
parts of every member of $A$ belong to $E$; hence $E=\Ree A$ and
\[
 A=E\oplus iE
\]
as real vector spaces.  Indeed, $E\cap iE=\{0\}$, because a function
which is both real-valued and purely imaginary-valued vanishes.  For
$a,b\in E$,
\begin{align*}
 \|a+ib\|_\infty
 &=\sup_{s\in S}\sqrt{a(s)^2+b(s)^2}\\
 &=\sup_{s\in S}\sup_{\theta\in\R}
   |a(s)\cos\theta-b(s)\sin\theta|\\
 &=\sup_{\theta\in\R}
   \|a\cos\theta-b\sin\theta\|_\infty.
\end{align*}
Thus the natural complex-linear map $E_{\C}\to A$, $a+ib\mapsto a+ib$,
is an onto isometry.  Composing it with $J^{-1}$ gives the required
complex-linear isometry $E_{\C}\to X$.
\end{proof}

\begin{lemma}\label{lem:real-form}
Let $E$ be a real Banach space.  If $E_{\C}$ is complex
$(1^+)$-injective, then $E$ is real $(1^+)$-injective.
\end{lemma}

\begin{proof}
Let $E$ be isometrically embedded in a real Banach space $Y$, and let
$\varepsilon>0$.  The complexification gives an isometric embedding
$E_{\C}\subseteq Y_{\C}$.  There is consequently a complex-linear
projection
\[
 P\colon Y_{\C}\longrightarrow E_{\C},
 \qquad \|P\|\leqslant1+\varepsilon.
\]
For $y\in Y$, put $Qy=\Ree P(y)$.  Formula \eqref{eq:taylor} shows that
the real-part map $E_{\C}\to E$ is contractive.  Hence
$\|Q\|\leqslant1+\varepsilon$, and $Qe=e$ for every $e\in E$.
Thus $Q$ is the required real projection.
\end{proof}

\begin{proof}[Proof of Theorem~\ref{thm:main}]
The assertion is clear for the zero space.  Suppose that $X\ne\{0\}$ is
a complex $(1^+)$-injective Banach space.  Embed $X$ isometrically in
$\ell_\infty(\Gamma)$ for a suitable set $\Gamma$.  For every
$\varepsilon>0$ there is a projection
\[
 P_\varepsilon\colon\ell_\infty(\Gamma)\longrightarrow X,
 \qquad \|P_\varepsilon\|\leqslant1+\varepsilon.
\]
The space $\ell_\infty(\Gamma)$ is $C(\beta\Gamma,\C)$, where
$\beta\Gamma$ is extremally disconnected, and is consequently a complex
$P_1$-space by \cite[Theorem~2]{Hasumi1958}.  Moreover,
$\lambda(X,\ell_\infty(\Gamma))=1$.  Theorem~\ref{thm:endpoint} now gives
that $X^*$ is isometric to an $L_1$-space and that $\ext B_X$ is non-empty.

Lemma~\ref{lem:hirsberg-lazar-real-form} gives a real Banach space
$E$ such that $X$ is complex-linearly isometric to $E_{\C}$.
Lemma~\ref{lem:real-form} shows that $E$ is real $(1^+)$-injective.
Theorem~\ref{thm:lindenstrauss-real-endpoint} shows that $E$ is real
$1$-injective.  By the real Goodner--Kelley--Nachbin theorem
\cite{Goodner1950,Kelley1952,Nachbin1950}, in the form recorded in
\cite[Theorem~6.8.3]{DalesEtAl2016}, there are an extremally disconnected
compact Hausdorff space $K$ and a real-linear isometry
\[
 E\cong C(K,\R).
\]
The norm computation in the proof of
Lemma~\ref{lem:hirsberg-lazar-real-form}, applied to $E=C(K,\R)$,
identifies its Taylor complexification with $C(K,\C)$.  Consequently,
\[
 X\cong E_{\C}\cong C(K,\C)
\]
complex-linearly and isometrically.
Finally, Hasumi's complex form of the Goodner--Kelley--Nachbin theorem
\cite[Theorem~2]{Hasumi1958} asserts that $C(K,\C)$ is $1$-injective.  Hence $X$ is
$1$-injective.
\end{proof}

\begin{proof}[Proof of Theorem~\ref{thm:classification}]
If $\lambda=1$, the asserted implication is
Theorem~\ref{thm:lindenstrauss-real-endpoint} over $\R$ and
Theorem~\ref{thm:main} over $\C$.
If $1<\lambda\leqslant2$, Theorem~A of
\cite{KaniaLewicki2023} supplies a $(\lambda^+)$-injective space
which is not $\lambda$-injective.  For $\lambda>2$, the same conclusion
is Theorem~A of \cite{KaniaLewicki2026}.
Both papers work over either fixed scalar field, and their
constructions yield closed subspaces of $\ell_\infty(\N,\mathbb K)$;
see also \cite[Remark~4.3]{KaniaLewicki2026}.
Thus \textup{(ii)} fails for every $\lambda>1$.
\end{proof}

\section*{Acknowledgements}

The first-named author thanks Fred Dashiell for his sustained interest
in this problem and for regularly enquiring about progress towards its
solution since 2013.


\begin{thebibliography}{99}

\bibitem{AP1956}
N.~Aronszajn and P.~Panitchpakdi,
\emph{Extension of uniformly continuous transformations and hyperconvex
metric spaces},
Pacific J. Math. \textbf{6} (1956), 405--439.
\url{https://doi.org/10.2140/pjm.1956.6.405}.

\bibitem{AvilesEtAl2016}
A.~Avil\'es, F.~Cabello S\'anchez, J.~M.~F. Castillo, M.~Gonz\'alez and
Y.~Moreno,
\emph{Separably injective Banach spaces},
Lecture Notes in Mathematics, vol.~2132, Springer, Cham, 2016.
\url{https://doi.org/10.1007/978-3-319-14741-3}.

\bibitem{DalesEtAl2016}
H.~G. Dales, F.~K. Dashiell, Jr., A.~T.-M. Lau and D.~Strauss,
\emph{Banach spaces of continuous functions as dual spaces},
CMS Books in Mathematics, Springer, Cham, 2016.
\url{https://doi.org/10.1007/978-3-319-32349-7}.

\bibitem{Dashiell2013}
F.~K. Dashiell, Jr.,
\emph{$(1+\varepsilon)$-injective Banach spaces, complex scalars},
MathOverflow, Question 132610, 2 June 2013.
\url{https://mathoverflow.net/q/132610}.

\bibitem{Effros1974}
E.~G. Effros,
\emph{On a class of complex Banach spaces},
Illinois J. Math. \textbf{18} (1974), 48--59.
\url{https://doi.org/10.1215/ijm/1256051348}.

\bibitem{Goodner1950}
D.~B. Goodner,
\emph{Projections in normed linear spaces},
Trans. Amer. Math. Soc. \textbf{69} (1950), no.~1, 89--108.
\url{https://doi.org/10.1090/S0002-9947-1950-0037465-6}.

\bibitem{Grothendieck1955}
A.~Grothendieck,
\emph{Une caract\'erisation vectorielle-m\'etrique des espaces $L^1$},
Canad. J. Math. \textbf{7} (1955), 552--561.
\url{https://doi.org/10.4153/CJM-1955-060-6}.

\bibitem{HWW1993}
P.~Harmand, D.~Werner and W.~Werner,
\emph{$M$-ideals in Banach spaces and Banach algebras},
Lecture Notes in Mathematics, vol.~1547, Springer-Verlag, Berlin, 1993.
\url{https://doi.org/10.1007/BFb0084355}.

\bibitem{Hasumi1958}
M.~Hasumi,
\emph{The extension property of complex Banach spaces},
Tohoku Math. J. (2) \textbf{10} (1958), 135--142.
\url{https://doi.org/10.2748/tmj/1178244708}.

\bibitem{HirsbergLazar1973}
B.~Hirsberg and A.~J. Lazar,
\emph{Complex Lindenstrauss spaces with extreme points},
Trans. Amer. Math. Soc. \textbf{186} (1973), 141--150.
\url{https://doi.org/10.1090/S0002-9947-1973-0333671-7}.

\bibitem{Hustad1973}
O.~Hustad,
\emph{A note on complex $P_1$-spaces},
Israel J. Math. \textbf{16} (1973), 117--119.
\url{https://doi.org/10.1007/BF02761976}.

\bibitem{Hustad1974}
O.~Hustad,
\emph{Intersection properties of balls in complex Banach spaces whose
duals are $L_1$ spaces},
Acta Math. \textbf{132} (1974), 283--313.
\url{https://doi.org/10.1007/BF02392118}.

\bibitem{KaniaLewicki2023}
T.~Kania and G.~Lewicki,
\emph{A forgotten theorem of Pe{\l}czy\'nski:
$(\lambda^+)$-injective spaces need not be $\lambda$-injective---the case
$\lambda\in(1,2]$},
Studia Math. \textbf{268} (2023), no.~3, 311--317.
\url{https://doi.org/10.4064/sm220119-25-6}.

\bibitem{KaniaLewicki2026}
T.~Kania and G.~Lewicki,
\emph{$(\lambda^+)$-injective Banach spaces},
J. Funct. Anal. \textbf{291} (2026), no.~11, Paper No.~111627.
\url{https://doi.org/10.1016/j.jfa.2026.111627}.

\bibitem{Kelley1952}
J.~L. Kelley,
\emph{Banach spaces with the extension property},
Trans. Amer. Math. Soc. \textbf{72} (1952), no.~2, 323--326.
\url{https://doi.org/10.1090/S0002-9947-1952-0045940-5}.

\bibitem{LimaScand1976}
{\AA}.~Lima,
\emph{An application of a theorem of Hirsberg and Lazar},
Math. Scand. \textbf{38} (1976), no.~2, 325--340.
\url{https://doi.org/10.7146/math.scand.a-11639}.

\bibitem{LimaIsrael1976}
{\AA}.~Lima,
\emph{Complex Banach spaces whose duals are $L_1$-spaces},
Israel J. Math. \textbf{24} (1976), no.~1, 59--72.
\url{https://doi.org/10.1007/BF02761429}.

\bibitem{LimaTAMS1977}
{\AA}.~Lima,
\emph{Intersection properties of balls and subspaces in Banach spaces},
Trans. Amer. Math. Soc. \textbf{227} (1977), 1--62.
\url{https://doi.org/10.1090/S0002-9947-1977-0430747-4}.

\bibitem{Lindenstrauss1964}
J.~Lindenstrauss,
\emph{Extension of compact operators},
Mem. Amer. Math. Soc., no.~48, American Mathematical Society, Providence,
RI, 1964.
\url{https://doi.org/10.1090/memo/0048}.

\bibitem{MunozEtAl1999}
G.~A. Mu{\~n}oz, Y.~Sarantopoulos and A.~Tonge,
\emph{Complexifications of real Banach spaces, polynomials and
multilinear maps},
Studia Math. \textbf{134} (1999), no.~1, 1--33.
\url{https://doi.org/10.4064/sm-134-1-1-33}.

\bibitem{Nachbin1950}
L.~Nachbin,
\emph{A theorem of the Hahn--Banach type for linear transformations},
Trans. Amer. Math. Soc. \textbf{68} (1950), no.~1, 28--46.
\url{https://doi.org/10.1090/S0002-9947-1950-0032932-3}.

\bibitem{Olsen1974}
G.~H. Olsen,
\emph{On the classification of complex Lindenstrauss spaces},
Math. Scand. \textbf{35} (1974), 237--258.
\url{https://doi.org/10.7146/math.scand.a-11550}.

\bibitem{Roy1981}
A.~K. Roy,
\emph{A short proof of a theorem on complex Lindenstrauss spaces},
Israel J. Math. \textbf{38} (1981), 41--45.
\url{https://doi.org/10.1007/BF02761846}.

\bibitem{Sakai1971}
S.~Sakai,
\emph{$C^*$-algebras and $W^*$-algebras},
Ergebnisse der Mathematik und ihrer Grenzgebiete, vol.~60,
Springer-Verlag, Berlin--New York, 1971.
\url{https://doi.org/10.1007/978-3-642-61993-9}.

\bibitem{Taylor1943}
A.~E. Taylor,
\emph{Analysis in complex Banach spaces},
Bull. Amer. Math. Soc. \textbf{49} (1943), no.~9, 652--669.
\url{https://doi.org/10.1090/S0002-9904-1943-07968-2}.

\end{thebibliography}
\end{document}